\documentclass[12pt]{article}
\usepackage{amsmath,amssymb,amsthm,latexsym,euscript,amscd}
\usepackage[english]{babel}
\usepackage[all]{xy}

\newtheorem{Theorem}{Theorem}
\newtheorem{Proposition}[Theorem]{Proposition}
\newtheorem{Corollary}[Theorem]{Corollary}
\newtheorem{Lemma}[Theorem]{Lemma}

\newcommand{\kk}{{\mathbb K}}
\newcommand{\LL}{{\mathbb L}}

\begin{document}

\title{Coherence of relatively quasi-free algebras}

\author{Alexey Bondal\thanks{ Steklov Institute of Mathematics, Moscow, and Kavli Institute for the Physics and Mathematics of the Universe (WPI), The University of Tokyo, Kashiwa, Chiba 277-8583, Japan,  and HSE Laboratory of algebraic geometry, Moscow, and The Institute for Fundamental Science, Moscow, and Bogolyubov Laboratory of Theoretical Physics, JINR, Dubna}~ and Ilya Zhdanovskiy\thanks{MIPT, Moscow, and HSE Laboratory of algebraic geometry, Moscow}}

\date{}

\maketitle

\section{Introduction}

An important problem in the representation theory of associative rings is to find conditions which guarantee that a noncommutative ring is coherent. Left coherence implies that the category of finitely presented left modules over the ring is abelian. This category might then be considered as being analogous to the category of coherent sheaves on an affine commutative variety. Thus, coherence is the initial point for developing noncommutative geometry in the framework of representation theory of rings in the style of the theory of coherent sheaves on algebraic varieties.

Despite its crucial importance, the problem of coherence is basically {\it terra incognita}: there are not so many classes of noncommutative rings with nice homological properties for which we can ensure coherence so far. Some recent references to results on coherence of various types of algebras the reader can find in the last paragraph of this paper. 

A class of algebras for which coherence is known is so-called quasi-free algebras over fields in the sense of Kuntz and Quillen \cite{CQ1}. We reproduce the proof of coherence for them in the main body of the text.

In this paper, we are interested in algebras that are relatively quasi-free over a commutative ring, say  $\kk$. By definition, this is an algebra $A$ over $\kk$ with the $A$-bimodule of noncommutative 1-forms $\Omega^1_{A/\kk }$ being projective.
Similar to the case of algebras that are quasi-free over a field one can show that relatively quasi-free algebras are exactly those
which satisfy the lifting property for nilpotent $\kk$-central extensions (proposition \ref{qfcriterion}).

We prove a theorem
that an algebra which is quasi-free relatively over a commutative {\em noetherian} ring $\kk$ is left and right coherent.

This result requires a different and more involved techniques in comparison to the case when the ring $\kk$ is a field. Our proof is based on the Chase criterion for coherence, which states that a ring is left coherent if and only if the product of any number of flat {\em right} modules is flat. In fact, the criterion is enough to check against the product of sufficiently many copies of rank one free modules over the algebra. In course of the proof, we use also a criterion of similar type for noetherian algebras.

It would be interesting to check whether the method of this paper is applicable to proving coherence for other important  classes of algebras. In \cite{BZ}, we explore algebras $B$, so-called {\em homotopes} constructed from an algebra $A$ and an element $\Delta\in A$. Algebra $B$ is defined by modifying multiplication in the vector space $A$ via the rule:
$$
a\cdot_{B}b=a\Delta b
$$
and adjoining a unit to the resulting algebra.
If element $\Delta$ verifies certain conditions, in which case it is dubbed well-tempered in \cite{BZ}, we prove that if the original algebra was coherent then so is the new one, thus providing an iterative way of constructing coherent algebras.


We are grateful to Dmitry Piontkovsky for useful references. This work was done during authors visit to Kavli IPMU and was supported by World Premier International Research Center Initiative (WPI Initiative), MEXT, Japan.

\section{Coherence of algebras and categories of finitely presented modules}


All rings and algebras which we consider in this paper are unital.

A left module $M$ over a ring is said to be {\em coherent} if it is finitely generated and for every morphism $\varphi : P\to M$ with free module $P$ of finite rank the kernel of $\varphi$ is finitely generated. A ring is {\em (left) coherent} if it is coherent as a left module over itself. Equivalently, a ring is coherent if every homomorphism between finitely generated projective modules over the ring has finitely generated kernel. If ring is coherent, then finitely presented modules are the same as coherent modules and the category of finitely presented modules is abelian \cite{Bou}.

Let $A$ be an algebra over a commutative ring $\kk$.
The definition due to Cuntz and Quillen \cite{CQ1} of a quasi-free algebra, which we adopt to the relative case (i.e. for algebras over a commutative ring $\kk$ rather than over a field), is that algebra $A$ over a field $k$ is {\em quasi-free} if the bimodule of noncommutative differential 1-forms
$\Omega^1_{A/k}$ (the definition is below)
is a projective $A\otimes_kA^{opp}$-module.

Every quasi-free algebra is {\em hereditary}, i.e. has global dimension $\le 1$. Indeed, for any two left $A$-modules $M$ and $N$ we have:
$$
{\rm Ext}^i_A(M, N)={\rm Ext}^i_{A-A}(A, {\rm Hom}_{k}(M, N)).
$$
The defining sequence (\ref{reldiff}) for the noncommutative differentials is a projective $A$-bimodule resolution of length 2 for $A$. It follows that ${\rm Ext}^i_A(M, N)=0$, for $i\ge 2$.

\begin{Lemma}\label{hereditary}
Hereditary rings are coherent.
\end{Lemma}

\begin{proof}
Indeed, a submodule of a projective module is projective for such rings. Thus, given a morphism $\varphi :P_1\to P_2$ between finitely generated projective modules, the image $I$ of $\varphi$ is a submodule of a projective module, hence projective too. Then short exact sequence induced by $\varphi$:
$$
0\to K\to P_1\to I\to 0,
$$
where $K$ is the kernel of $\varphi$, splits. Hence, we have an epimorphism $P_1\to K$, which proves that $K$ is finitely generated.
\end{proof}

This implies that every quasi-free algebra is left (and, similarly, right) coherent.

It follows from the proof that the number of generators for the kernel $K$ as a left $A$-module is bounded by the number of generators for $P_1$, which shows how special is this case. In general, even noetherian commutative rings might be of arbitrary global dimension and the number of generators of the kernel of a morphism between two finite rank free modules is not restricted by any function on ranks of the two free modules. For that reason this proof of coherence does not seem to be applicable to the case of relatively quasi-free algebras that we consider below. We will need a substantially different argument.


Let us develop a bit of the theory for the case of relatively quasi-free algebras. We consider an algebra $A$ over a commutative ring $\kk$.
An $A$-{\em bimodule} is said to be $\kk $-{\em central} if left and right $\kk $ actions coincide. Such $A$-bimodules are identified with left $A\otimes_{\kk} A^{opp}$-modules.

Define the $A\otimes_{\kk} A^{opp}$-module of relative noncommutative 1-forms by a short exact sequence
\begin{equation}\label{reldiff}
0\to \Omega^1_{A/\kk}\to A\otimes_{\kk}A\to A\to 0
\end{equation}

{\bf Definition.} Let $\kk$ be a commutative ring. An algebra $A$ over $\kk$ is said to be quasi-free over $\kk$ if $\Omega^1_{A/\kk}$ is a projective $A\otimes_{\kk} A^{opp}$-module.

Let $A$ be a relatively quasifree algebra over $\kk$.
By applying functor ${\rm Ext}^i_{A\otimes_{\kk} A^{opp}}(-, M)$ to the exact sequence (\ref{reldiff}) for an arbitrary $\kk $-central $A$-bimodule $M$, we obtain:
\begin{equation}\label{ext2=0}
{\rm Ext}^2_{A\otimes_{\kk} A^{opp}}(A, M)={\rm Ext}^1_{A\otimes_{\kk} A^{opp}}(\Omega^1_{A/\kk}, M)=0.
\end{equation}

A generalization of Cuntz-Quillen criterion  for quasi-free algebras \cite{CQ1} holds in the relative case too.
\begin{Proposition}
\label{qfcriterion}
Algebra $A$ is quasi-free over $\kk$ if and only if for any ${\tilde R}$, a nilpotent extension of an algebra $R$ by a square zero ideal   in the category of algebras over $\kk$, and any homomorphism $A\to R$ there exists its lifting to a homomorphism $A\to {\tilde R}$.
\end{Proposition}
\begin{proof} Let $A$ be a quasi-free algebra. Consider a square-zero nilpotent extension ${\tilde R}\to R$ and a homomorphism $A\to R$. Denote by $I$ the kernel of ${\tilde R}\to R$. By assumptions, it is a square zero ideal in ${\tilde R}$, hence it has a natural structure of $\kk$-central $R$-bimodule. Let ${\tilde A}$ be the fibred product over $R$ of $A$ and  ${\tilde R}$. This is a $\kk$-algebra, which is a square zero extension of $A$ by $I$, where $I$ is endowed with a $\kk $-central $A$-bimodule structure which is the pull-back of $R$-bimodule structure. Such extensions are classified by ${\rm Ext}^2_{A\otimes_{\kk} A^{opp}}(A, I)$. This group is trivial for quasi-free algebras by (\ref{ext2=0}). Hence, we have a splitting homomorphism  $A\to {\tilde A}$. When combined with the map ${\tilde A} \to {\tilde R}$ it gives the required lifting.

Conversely, if all square zero extensions allow liftings, then, by taking $R=A$ and $I$ arbitrary $\kk $-central $A$-bimodule, we see that ${\rm Ext}^2_{A\otimes_{\kk} A^{opp}}(A, I)={\rm Ext}^1_{A\otimes_{\kk} A^{opp}}(\Omega^1_{A/\kk}, I)=0$, i.e. $\Omega^1_{A/\kk}$ is a projective $A\otimes_{\kk} A^{opp}$-bimodule.

\end{proof}

\section{Coherence for relatively quasi-free algebras}

Recall the following criterion of coherence due to Chase \cite{Chase}.
\begin{Lemma}\label{Chasecriterion} For any ring $A$ the following are equivalent:
\begin{itemize}
\item $A$ is left coherent,
\item For any family of {\em right} flat modules $F_i$, $i\in I$, the product $\prod_{i\in I}F_i$ is right flat,
\item For the family $F_i\cong A$ of free modules with ${\rm card}\ I= {\rm card}\ A$, the product $\prod_{i\in I}F_i$ is right flat.
\end{itemize}
\end{Lemma}

We will use also a criterion in the same style for noetherianess (cf. \cite{Ab}).
\begin{Lemma}\label{Ncriterion} For any ring $A$ the following are equivalent:
\begin{itemize}
\item $A$ is left noetherian,
\item For any left $A$-module $M$ and any family of flat {\em right} modules $F_i$, $i\in I$, the morphism $(\prod_{i\in I}F_i)\otimes_AM\to \prod_{i\in I}(F_i\otimes_AM)$ is mono,
\item For any left $A$-module $M$ and the family of rank 1 free {\em right} modules $F_i\cong A$, $i\in I$, ${\rm card}\ I= {\rm card}\ A$, the morphism $(\prod_{i\in I}F_i)\otimes_AM\to \prod_{i\in I}(F_i\otimes_AM)$ is mono.
\end{itemize}
\end{Lemma}
\begin{Theorem}\label{thquasifree}
Let $\kk$ be a commutative noetherian ring. Assume that $A$ is an algebra quasi-free relatively over $\kk$ and $A$ is flat as a $\kk$-module. Then $A$ is a left and right coherent algebra.
\end{Theorem}
\begin{proof}
Consider left $A$-module $M$ and a family of rank 1 free right $A$-modules $F_i\cong A$. We shall consider all left (respectively, right) $A$-modules to be always endowed with the right (respectively, left) $\kk$-module structure identical to the left (respectively, right) $\kk$-module structure.

We have a natural transformation of functors on the category of $A$-bimodules with values in the category of $\kk$-modules:
$$
(-)\otimes _{A\otimes_{\kk} A^{opp}}(M\otimes_{\kk}\prod_i F_i)\to \prod_i((-)\otimes _{A\otimes_{\kk} A^{opp}}(M\otimes_{\kk} F_i))
$$
Let us apply functors $(-)\otimes _{A\otimes_{\kk} A^{opp}}(M\otimes_{\kk}\prod_i F_i)$ and $\prod_i((-)\otimes _{A\otimes_{\kk} A^{opp}}(M\otimes_{\kk} F_i))$ and their derived functors to the exact sequence (\ref{reldiff}). The above natural transformation provides us with a commutative diagram with exact rows:
\begin{equation}\label{tordiagram}
\xymatrix{0 \ar[r] & {\rm Tor}^{A\otimes_{\mathbb K}A^{opp}}_1(A, M\otimes_{\mathbb K} \prod_i F_i)\ar[r]\ar[d] & \Omega^1_{A/{\mathbb K}} \otimes_{A \otimes_{\mathbb K}A^{opp}}(M \otimes_{\mathbb K} \prod_i F_i)\ar[d]
\\
0\ar[r] & \prod_i {\rm Tor}^{A \otimes_{\mathbb K} A^{opp}}_1(A, M\otimes_{\mathbb K}F_i)\ar[r] & \prod_i (\Omega^1_{A/{\mathbb K}} \otimes_{A \otimes_{\mathbb K} A^{opp}} (M \otimes_{\mathbb K} F_i))}
\end{equation}
For any left $A$-module $M$ and right $A$-module $N$, we have an isomorphism of objects in the derived categories:
$$
A\otimes_{A\otimes_{\kk} A^{opp}}^{\LL}(M\otimes_{\kk}^{\LL}N)=N\otimes^{\LL}_AM,
$$
which implies a spectral sequence:
$$
{\rm Tor}_i^{A\otimes_{\kk} A^{opp}}(A, {\rm Tor}_j^{\kk}(M,N))\Longrightarrow {\rm Tor}^A_{i+j}(N, M)
$$
For a flat $\kk$-module $N$, it implies that
\begin{equation}\label{tortor}
{\rm Tor}_1^{A\otimes_{\kk} A^{opp}}(A, M\otimes_{\kk}N)={\rm Tor}^A_{1}(N, M).
\end{equation}
Since $F_i$ are rank 1 free as $A$-modules and $A$ is a flat $\kk$-module, $F_i$ are also flat $\kk$-modules. Hence, ${\rm Tor}_1^{A\otimes_{\kk} A^{opp}}(A, M\otimes_{\kk}F_i)={\rm Tor}_1^A(F_i, M)=0$.

Since $\kk$ is noetherian, it is also coherent, hence by Chase criterion, lemma  \ref{Chasecriterion},
$\prod_i F_i$ is also $\kk$-flat. In view of (\ref{tortor}), diagram (\ref{tordiagram}) reads:
\begin{equation}\label{tordiagram1}
\xymatrix{0 \ar[r] & {\rm Tor}^{A}_1(\prod_i F_i, M)\ar[r]\ar[d] & \Omega^1_{A/{\mathbb K}} \otimes_{A \otimes_{\mathbb K}A^{opp}}(M \otimes_{\mathbb K} \prod_i F_i)\ar[d]
\\
 & 0\ar[r] & \prod_i (\Omega^1_{A/{\mathbb K}} \otimes_{A \otimes_{\mathbb K} A^{opp}} (M \otimes_{\mathbb K} F_i))}
\end{equation}
Let us show that $\Omega^1_{A/\kk}\otimes_{A\otimes_{\kk} A^{opp}}(M\otimes_{\kk}\prod_i F_i)\to \prod_i (\Omega^1_{A/\kk}\otimes_{A\otimes A^{opp}}(M\otimes_{\kk} F_i))$ is an embedding.

First, consider the case when $D=\oplus_j(A\otimes_{\kk}A^{opp})$ is a free $A\otimes_{\kk}A^{opp}$-module. Then
$$
D\otimes_{A\otimes_{\kk}A^{opp}}(M\otimes_{\kk}\prod_iF_i)=\oplus_j(M\otimes_{\kk}\prod_iF_i)
$$
and
$$\prod_i(D\otimes_{A\otimes_{\kk}A^{opp}}(M\otimes_{\kk}F_i))=\prod_i(\oplus_j(M\otimes_{\kk}F_i)).
$$
The morphism
$$
D\otimes_{A\otimes_{\kk}A^{opp}}(M\otimes_{\kk}\prod_iF_i)\to \prod_i(D\otimes_{A\otimes_{\kk}A^{opp}}(M\otimes_{\kk}F_i))
$$
is the composite of two morphism:
$$
\oplus_j(M\otimes_{\kk}\prod_iF_i)\to \oplus_j\prod_i(M\otimes_{\kk}F_i)\to \prod_i(\oplus_j(M\otimes_{\kk}F_i)).
$$
By criterion of noetherianess, lemma  \ref{Ncriterion}, for $\kk$, we have that $M\otimes_{\kk}\prod_iF_i\to \prod_i(M\otimes _{\kk}F_i)$ is an embedding. Therefore, the first morphism is an embedding. The second morphism is readily an embedding, as the morphism $\oplus_j\prod_i\to \prod_i\oplus_j$ is so. Hence the composite is an embedding.

Since $\Omega^1_{A/\kk}$ is a projective bimodule, we have an imbedding
$$
\Omega^1_{A/\kk}\to \oplus_jA\otimes_{\kk}A^{opp}
$$
as a direct summand. This implies a diagram:
\begin{equation}\label{omegaembedding}
\xymatrix{\Omega^1_{A/{\mathbb K}} \otimes_{A \otimes_{\mathbb K}A^{opp}}(M \otimes_{\mathbb K} \prod_i F_i) \ar[r]\ar[d] & \oplus_j(M \otimes_{\mathbb K}\prod_i F_i)\ar[d]\\
\prod_i (\Omega^1_{A/{\mathbb K}} \otimes_{A \otimes_{\mathbb K} A^{opp}} (M \otimes_{\mathbb K} F_i)) \ar[r] & \prod_i(\oplus_j (M \otimes_{\mathbb K} F_i))}
\end{equation}
The upper horizontal arrow is an embedding because it is obtained by taking tensor product of an embedding of a direct summand with  module $M\otimes_{\kk}\prod_i F_i$. We have shown above that the right vertical arrow is injective. Therefore, the
left vertical arrow is an embedding.

Then, diagram (\ref{tordiagram1}) implies that ${\rm Tor}_1^A(\prod_iF_i , M)=0$, i.e $\prod_iF_i$ is a right flat module. By Chase criterion, lemma \ref{Chasecriterion}, algebra $A$ is left coherent.

Right coherence follows similarly.

\end{proof}

Similarly, one can prove the following statement.

\begin{Theorem} Let algebra $A$ over a commutative ring $\kk$ be of Hochschild dimension $n$. Then Tor-dimension of any product of free right modules is less than $n$.
\end{Theorem}
\begin{proof}
Consider a projective $A$-bimodule resolution for $A$ of length $n+1$. Let $M$ be a left $A$-module and $F_i$ free right $A$-modules. By applying functors $(-)\otimes _{A\otimes_{\kk} A^{opp}}(M\otimes_{\kk}\prod_i F_i)$ and $\prod_i((-)\otimes _{A\otimes_{\kk} A^{opp}}(M\otimes_{\kk} F_i))$ to the resolution we get two complexes. Moreover, the natural transformation between the functors provides a morphism from the first complex to the second one. The same argument as in the proof of the previous theorem shows that the second complex is exact and the first one is termwise embedded into the second one. Hence, we get an exact triangle of complexes, where the third term is the quotient of the second complex by the first one. It implies long exact sequence on cohomology of these three complexes, that shows that cohomology of the first complex in the most left term is trivial. This cohomology is identified with ${\rm Tor}_n^A(\prod_iF_i , M)$. This proves the theorem.
\end{proof}

\section{Examples and related results}

If $A$ is quasi-free over a field $k$ and $\kk$ is a commutative $k$-algebra, then $A\otimes_k \kk$ is flat and quasi-free over $\kk$. This gives numerous examples where theorem \ref{thquasifree} is applicable, if we take $\kk$ to be a noetherian ring. Among those are free algebras, path algebras of quivers with no relation, etc.

The criterion of coherence in terms of lifting the nilpotents, proposition \ref{qfcriterion}, implies that coherence is preserved under localizations of algebras, which shows, in particular, that group algebras of free groups over noetherian rings $\kk$ are coherent.

More examples come if we consider algebras which are Morita equivalent to examples, considered above, because coherence is preserved under Morita equivalence. Thus, we have a

\begin{Corollary}
The matrix algebra over the group algebra of a free group over a noetherian ring is coherent.
\end{Corollary}

This example is of interest when considering perverse sheaves and harmonic analysis on graph (cf. \cite{BZ}).

There are few results on coherence of algebras. In \cite{CLL}, the authors proved that if $\kk$ is a noetherian algebra and $A$ and $B$ are augmented coherent $\kk$ algebras, then the coproduct $A*B$ over $\kk$ is coherent too. Aberg in \cite{Ab} was able to remove the assumption that algebras $A$ and $B$ are augmented by heavy use of Chase criterion.

It makes sense to compare Chase criterion, used in this paper, with another cohomological approach applicable to the case of local algebras.
There is a simple way to control coherence of local algebras (cf. \cite{Po}).
A finitely generated local algebra is coherent if, for every finitely generated ideal $J$, the space ${\rm Tor}^A_1(k, J)$ is finite dimensional over the residue field $k$. In particular, it allows to prove that if a graded algebra has a double-sided ideal $J$, such that the quotient algebra is right notherian, and $J$ is free as a left $A$-module, then $A$ is right coherent (\cite{Po}, \cite{P}). This criterion allows to prove coherence for some classes of algebras, see \cite{S}, \cite{H}, \cite{HOZ}.

An interesting example of coherence, which is beyond the scope of methods of the current paper, as well as the above criterion for local algebras, is the one due to D. Piontkovski \cite{P}. He showed coherence of the {\em geometric} ${\mathbb Z}$-algebra of the helix with the thread consisting of 3 elements $({\cal O}(-1), {\cal O}, {\cal O}(1))$ of an exceptional collection on ${\mathbb P}^3$ (see \cite{BP}). Note that any full exceptional collection on ${\mathbb P}^3$ has 4 elements. Coherence of the algebra implies the existence of a t-structure in the subcategory generated by the 3 elements of the exceptional collection, similar to the standard geometric t-structure in the derived category of coherent sheaves on ${\mathbb P}^3$. This subcategory is interesting in relation to the problem of description of mathematical instantons on ${\mathbb P}^3$, because it contains instanton vector bundles of arbitrary topological charge. In \cite{ELO}, this subcategory is related to a noncommutative Grassmannian.


Note that the condition on the commutative ring $\kk$ to be noetherian cannot be removed in the statement of the theorem \ref{thquasifree}. A counterexample due to Soublin is known, where $A$ is a polynomial algebra in one variable over a coherent commutative ring $\kk$ \cite{Sou}.
More precisely, if $\kk$ is the direct product of a countable number of copies of ${\mathbb Q}[[x, y]]$, the ring of formal power series in two variables, then the ring $\kk [t]$ is not coherent. Note that the ring $\kk$ in this case is even {\em uniformly coherent} and has no nilpotents.

For graded rings, there is a weaker notion of graded coherence, when the category of graded modules is considered. H. Minamoto showed in \cite{Mi} that, in contrast to the non-graded case, algebra $A[t]$ is graded coherent, if $A$ is coherent and degree of $t$ is 1.

\def\cprime{$'$}
\ifx\undefined\bysame
\newcommand{\bysame}{\leavevmode\hbox to3em{\hrulefill}\,}
\fi

\end{document}